\documentclass[10pt]{article}
\usepackage{tkz-graph,amssymb,amsfonts,enumerate,mathtools}

\newtheorem{prethm}{{\bf  Theorem}}

\newenvironment{thm}{\begin{prethm}{\hspace{-0.5
               em}{\bf .}}}{\end{prethm}}

\newtheorem{prepro}{{\bf  Theorem}}

\newtheorem{precor}{{\bf  Corollary}}

\newenvironment{cor}{\begin{precor}{\hspace{-0.5
               em}{\bf .}}}{\end{precor}}
\newtheorem{preconj}{{\bf  Conjecture}}

\newtheorem{preremark}{{\bf  Remark}}

\newtheorem{preexample}{{\bf  Example}}

\newtheorem{preexam}{{\bf  Example}}

\newenvironment{exa*}{\begin{preexam}{\hspace{-0.5
               em}{\bf .}}}{\end{preexam}}

\newtheorem{prelem}{{\bf  Lemma}}

\newenvironment{lem}{\begin{prelem}{\hspace{-0.5
               em}{\bf .}}}{\end{prelem}}
\newtheorem{preproof}{{\bf  Proof.}}

\newenvironment{proof}[1]{\begin{preproof}{\rm
               #1}\hfill{$\Box$}}{\end{preproof}}

\title{\large \bf  The Prime Index Graph of a Group}

\author{\bf\small\sc
S. Akbari$^{a}$, A. Ashtab$^{a}$, F. Heydari
\thanks
{
Corresponding author.\newline \indent
~
\newline \indent
~~{2010{ \it Mathematics Subject Classification}: 05C25, 05C40, 20E99.}
\newline \indent
~~{\it E-mails}: s\_akbari@sharif.edu, ashtab\_arman@mehr.sharif.ir, f-heydari@kiau.ac.ir,
\newline \indent ~~rezai.mar@gmail.com, fa.sherafati@gmail.com.}
$^{,b}$, M. Rezaee$^{b}$, F. Sherafati$^{b}$\\
{ \footnotesize {\em $^{\rm a}$Department of Mathematical Sciences,
Sharif University of Technology, Tehran, Iran}}\\
{ \footnotesize {\em $^{\rm b}$Department of Mathematics, Karaj Branch,
Islamic Azad University, Karaj, Iran}}}

\date{}

\begin{document}
\maketitle
\begin{abstract}
Let $G$ be a group. The prime index graph of $G$, denoted by $\Pi(G)$, is the graph whose vertex set is the set of all subgroups of $G$ and two distinct comparable vertices $H$ and $K$ are adjacent if and only if the index of $H$ in $K$ or the index of $K$ in $H$ is prime. In this paper, it is shown that for every group $G$, $\Pi(G)$ is bipartite and the girth of $\Pi(G)$ is contained in the set $\{4,\infty\}$. Also we prove that if $G$ is a finite solvable group, then $\Pi(G)$ is connected.
\end{abstract}

\section{Introduction}

\indent Let $\Gamma$ be a graph. We say that $\Gamma$ is {\it connected} if there is a path between any two distinct
vertices of $\Gamma$.
We denote by $d(v)$, the degree of a vertex $v$ in $\Gamma$. A graph in which every vertex has the same degree is called a \textit{regular graph}. If all vertices have degree $k$, then the graph is said to be \textit{$k$-regular}.
The \textit{girth} of $\Gamma$, denoted by $gr(\Gamma)$, is the length of a shortest cycle in $\Gamma$ (We say that $gr(\Gamma)=\infty$ if $\Gamma$ contains no cycle).
A \textit{null graph} is a graph with no edges. A \textit{forest} is a graph with no cycle.
We denote the complete graph, the path and the cycle
of order $n$ by $K_n$, $P_n$ and $C_n$, respectively. We use \textit{$n$-cycle} to denote the cycle of order $n$, where $n \geq 3$. The \textit{Cartesian product} of two graphs $\Gamma$ and $\Omega$ is denoted by $\Gamma\square \Omega$. The \textit{hypercube graph} $Q_s$ is the Cartesian product of $s$ copies of $P_2$.\\
\indent Let $G$ be a group. We denote the identity element of $G$ by $e$. The derived subgroup of $G$ is denoted by $G'$ and $G^{(n+1)}=(G^{(n)})'$, where $n$ is a positive integer. For any subgroup $H$ of $G$, the intersection of all the conjugates of $H$ in $G$ is denoted by $\mathrm{Core}_G(H)$. Let $x\in G$. Then the subgroup generated by $x$ is denoted by $\langle x\rangle$.
As usual, $\mathbb{Z}_n$, $A_n$ and $S_n$ denote the group of integers modulo $n$, the alternating group and the symmetric group of degree $n$,  respectively. For a fixed prime $p$, the \textit{quasicyclic $p$-group} is denoted by $\mathbb{Z}{(p^\infty)}$. Also the projective special linear group of degree $n$ over the field $\mathbb{Z}_p$ is denoted by PSL$(n,p)$.
\\
\indent There are several graphs associated with groups, for instance non-commuting graph of a group, intersection graph of subgroups of a group, and subgroup graph of a group. (See \cite{akbari,hey,erfa}.) The subgroup graph of a group $G$ is defined as the
graph of its lattice of subgroups, that is, the graph whose vertices are the subgroups
of $G$ such that two subgroups $H$ and $K$ are adjacent if one of $H$ or $K$ is maximal in
the other. In this article, we introduce and investigate the \textit{prime index graph} of $G$, denoted by $\Pi(G)$. It is an undirected graph whose vertices are all subgroups of $G$
and two distinct comparable
vertices $H$ and $K$ are adjacent if and only if $[H : K]$ or $[K : H]$ is prime. Clearly, the prime index graph of $G$ is a subgraph of the subgroup graph of $G$ and whenever $G$ is a nilpotent group, see \cite[p.143]{scot}, then these two graphs are coincide.
In follows the prime index graphs of $S_3$ and $A_4$ are given. Note that $H\cong \mathbb{Z}_2\times \mathbb{Z}_2$, $K_i\cong \mathbb{Z}_2$, for $i=1,2,3$ and $L_j\cong \mathbb{Z}_3$, for $j=1,\ldots,4$.
{\rm
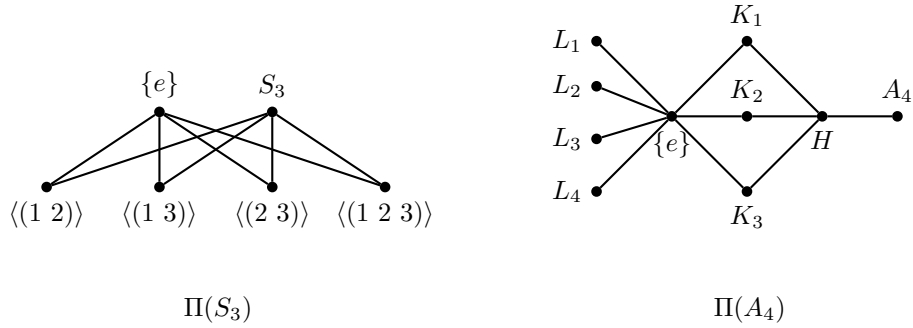
\begin{center}
    \begin{tikzpicture}
        \GraphInit[vstyle=Classic]
        \Vertex[x=0,y=0,style={black,minimum size=3pt},LabelOut=true,Lpos=270,L=$\langle(1\ 2)\rangle$]{3}
        \Vertex[x=1.5,y=0,style={black,minimum size=3pt},LabelOut=true,Lpos=270,L=$\langle(1\ 3)\rangle$]{4}
        \Vertex[x=3,y=0,style={black,minimum size=3pt},LabelOut=true,Lpos=270,L=$\langle(2\ 3)\rangle$]{5}
        \Vertex[x=4.5,y=0,style={black,minimum size=3pt},LabelOut=true,Lpos=270,L=$\langle(1\ 2\ 3)\rangle$]{6}
        \Vertex[x=1.5,y=1,style={black,minimum size=3pt},LabelOut=true,Lpos=90,L=$\{e\}$]{1}
        \Vertex[x=3,y=1,style={black,minimum size=3pt},LabelOut=true,Lpos=90,L=$S_3$]{2}
        \Edges(1,3)
        \Edges(1,4)
        \Edges(1,5)
        \Edges(1,6)
        \Edges(2,3)
        \Edges(2,4)
        \Edges(2,5)
        \Edges(2,6)
    \end{tikzpicture}
\hspace{1.1cm}
    \begin{tikzpicture}
        \GraphInit[vstyle=Classic]
        \Vertex[x=0,y=2,style={black,minimum size=3pt},LabelOut=true,Lpos=180,L=$L_1$]{1}
        \Vertex[x=0,y=1.4,style={black,minimum size=3pt},LabelOut=true,Lpos=180,L=$L_2$]{2}
        \Vertex[x=0,y=.7,style={black,minimum size=3pt},LabelOut=true,Lpos=180,L=$L_3$]{3}
        \Vertex[x=0,y=0,style={black,minimum size=3pt},LabelOut=true,Lpos=180,L=$L_4$]{4}
        \Vertex[x=1,y=1,style={black,minimum size=3pt},LabelOut=true,Lpos=270,L=$\{e\}$]{5}
        \Vertex[x=2,y=2,style={black,minimum size=3pt},LabelOut=true,Lpos=90,L=$K_1$]{6}
        \Vertex[x=2,y=1,style={black,minimum size=3pt},LabelOut=true,Lpos=90,L=$K_2$]{7}
        \Vertex[x=2,y=0,style={black,minimum size=3pt},LabelOut=true,Lpos=270,L=$K_3$]{8}
        \Vertex[x=3,y=1,style={black,minimum size=3pt},LabelOut=true,Lpos=270,L=$H$]{9}
        \Vertex[x=4,y=1,style={black,minimum size=3pt},LabelOut=true,Lpos=90,L=$A_4$]{10}
        \Edges(10,9)
        \Edges(9,8)
        \Edges(9,7)
        \Edges(9,6)
        \Edges(8,5)
        \Edges(7,5)
        \Edges(6,5)
        \Edges(5,1)
        \Edges(5,2)
        \Edges(5,3)
        \Edges(5,4)
    \end{tikzpicture}
\end{center}
\hspace{4.2cm}
$\Pi(S_3)$ \hspace{5.9cm}
$\Pi(A_4)$
}\\

Here we show that for every group $G$, $\Pi(G)$ is a bipartite graph and $gr(\Pi(G))\in \{4, \infty\}$. We prove that for any finite abelian group $G$, $\Pi(G)$ is a regular graph if and only if $\Pi(G)$ is a hypercube graph. Finally, we study the connectivity of $\Pi(G)$ and we show that for every finite solvable group $G$, $\Pi(G)$ is a connected graph. Among other results, we prove that if $\Pi(G)$ is a connected graph and $N$ is a normal subgroup of $G$, then both graphs $\Pi(N)$ and $\Pi(G/N)$ are connected.

\section{The Prime Index Graphs are Bipartite}

In this section, we show that the prime index graph of a group $G$ is bipartite. To see this, we prove a stronger result.
First we define a directed graph $\overrightarrow{\Gamma}(G)$.
It is a directed graph whose vertex set is the set of all subgroups of $G$ and for every two distinct vertices $H$ and $K$, there is an arc from $H$
to $K$, whenever $H\subseteq K$ and $[K : H]=r$, for some positive integer $r$. Suppose that $r$ is the weight of the arc from $H$ to $K$.

\begin{thm}
\label{ashtab}
Let $C$ be a cycle of $\overrightarrow{\Gamma}(G)$. Then the product of weights of all clockwise arcs of $C$ is equal to the product of weights of all counter-clockwise arcs of $C$.
\end{thm}

\begin{proof}
{Let $C$ be a cycle of $\overrightarrow{\Gamma}(G)$ of length $n$. We prove the theorem by induction on $n$. Clearly, for $n = 3$ the assertion holds. Now, suppose that $n>3$ and the assertion is true for every integer $m$, $3\leq m<n$. If $C$ contains a directed path $P$ of length 2, such as $H\xrightarrow{r} K\xrightarrow{s} L$, then we replace $P$ with the path $H\xrightarrow{rs} L$. Hence by the induction hypothesis the result holds. Otherwise, $C$ contains a path of the form $H\xrightarrow{r} K\xleftarrow{s} L\xrightarrow{t} M$. We consider two cases:

{\bf Case 1}. If $H\cap L$ is not a vertex of $C$, then
we replace $H\xrightarrow{r} K\xleftarrow{s} L$ with the path $H\xleftarrow{s'} H\cap L\xrightarrow{r'} L$, where $[L : H\cap L]=r'$ and $[H : H\cap L]=s'$. Note that $s'r=r's$ and so $r/s=r'/s'$.
Next, we replace $H\cap L\xrightarrow{r'} L\xrightarrow{t} M$ with $H\cap L\xrightarrow{r't} M$. Thus we find a cycle $C_1$ of length $n-1$ and by the induction hypothesis $r'ta=s'b$, where $a$ is the product of weights of all clockwise arcs of $C_1$ except the weight of $H\cap L\xrightarrow{r't} M$ and $b$ is the product of weights of all counter-clockwise arcs of $C_1$ except the weight of $H\xleftarrow{s'} H\cap L$. Hence $r/s=r'/s'=b/ta$ and so $rta=sb$. It is clear that $rta$ is the product of weights of all clockwise arcs of $C$ and $sb$ is the product of weights of all counter-clockwise arcs of $C$. The result holds.

{\bf Case 2}. Assume that $H\cap L$ is a vertex of $C$. Clearly, $H\cap L\neq H$ or $H\cap L\neq L$. With no loss of generality, suppose that
$H\cap L\neq H$. By adding the arc $H\cap L \xrightarrow{} H$, we find two cycles $C_1$ and $C_2$ of lengths less than $n$. Let $[H : H\cap L]=s'$. Assume that the arc $H\cap L \xrightarrow{s'} H$ is a clockwise arc of $C_1$. So $H\cap L \xrightarrow{s'} H$ is a counter-clockwise arc of $C_2$. Now, by the induction hypothesis, $s'=b_1/a_1=a_2/b_2$, where $a_1$ is the product of weights of all clockwise arcs of $C_1$ except the weight of $H\cap L \xrightarrow{s'} H$, $a_2$ is the product of weights of all clockwise arcs of $C_2$, $b_1$ is the product of weights of all counter-clockwise arcs of $C_1$, and $b_2$ is the product of weights of all counter-clockwise arcs of $C_2$ except the weight of $H\cap L \xrightarrow{s'} H$. Thus $a_1a_2=b_1b_2$. The proof is complete.
}
\end{proof}

Now, we are in a position to prove the following corollary.

\begin{cor}
\label{bipartite}
Let $G$ be a group. Then $\Pi(G)$ is bipartite.
\end{cor}

\begin{proof}
{We show that every cycle of $\Pi(G)$ is an even cycle. If $\Pi(G)$ has a cycle $C$, we may assume that $C$ is a cycle in $\overrightarrow{\Gamma}(G)$. Now, by Theorem \ref{ashtab}, since all weights are primes, the number of clockwise arcs of $C$ is equal to the number of counter-clockwise arcs of $C$. Hence $C$ is an even cycle. This implies that $\Pi(G)$ is a bipartite graph.
}
\end{proof}

If $G$ is a non-trivial group and $e\neq x\in G$, then $\langle x \rangle$ contains a subgroup of prime index and hence $d(\langle x \rangle)\geq 1$. So $\Pi(G)$ is not a null graph.

\begin{lem}
Let $G$ be a group. Then $\Pi(G)$ is a complete bipartite graph if and only if $G$ is a cyclic group of prime order or $|G|=pq$, for some primes $p$ and $q$.
\end{lem}

\begin{proof}
{Clearly, if $G\cong \mathbb{Z}_p$, then $\Pi(G)\cong K_2$. Also if $|G|=pq$, then $\Pi(G)$ is a complete bipartite graph whose one part contains all subgroups of $G$ of orders $p$ or $q$ and the other part contains $\{e\}$ and $G$. Conversely, assume that $\Pi(G)$ is complete bipartite. If $\{e\}$ and $G$ are contained in two different parts of $\Pi(G)$, then $G\cong \mathbb{Z}_p$, where $p$ is a prime number. Otherwise, there exists a subgroup $H$ of $G$ adjacent to both $\{e\}$ and $G$. Thus $|G|=pq$, for some primes $p$ and $q$.
}
\end{proof}

The following theorem shows that if $\Pi(G)$ contains a cycle $C$, then $gr(\Pi(G))=4$.

\begin{thm}
\label{girth}
Let $G$ be a group. Then $gr(\Pi(G))\in \{4, \infty\}$.
\end{thm}

\begin{proof}
{First assume that $G$ is finite and $|G|=p^{n_1}_1\cdots p^{n_s}_s$, where $p_1,\ldots,p_s$ are distinct primes and $n_1,\ldots,n_s$ are positive integers. Suppose that $L_i$ is a Sylow $p_i$-subgroup of $G$, for $i=1,\ldots,s$. First assume that $L_i$ contains two distinct maximal subgroups $H$ and $K$, for some $i$. Since $H$ and $K$ are normal subgroups of $L_i$, so $HK=L_i$. This implies that $|H\cap K|=p^{n_i-2}_i$ and hence $L_i$ -- $H$ -- $H\cap K$ -- $K$ -- $L_i$ is a 4-cycle in $\Pi(G)$. So by Corollary \ref{bipartite}, $gr(\Pi(G))=4$. Next, assume that $L_i$ contains a unique maximal subgroup, for $i=1,\ldots,s$. Hence all Sylow subgroups of $G$ are cyclic.  Now, by \cite[Theorem 10.26]{rose}, $G$ is a supersolvable group. If $s\geq 2$, then $G$ has a subgroup $K$ of order $p_1p_2$ (\cite[p.292]{rose}). Let $H_i$ be a subgroup of $K$ of order $p_i$, for $i=1,2$. Hence $\{e\}$ -- $H_1$ -- $K$ -- $H_2$ -- $\{e\}$ is a 4-cycle in $\Pi(G)$ and so by Corollary \ref{bipartite}, $gr(\Pi(G))=4$. If $s=1$, then $G\cong \mathbb{Z}_{p^{n_1}_1}$. Thus $\Pi(G)\cong P_{n_1+1}$ and $gr(\Pi(G))=\infty$.\\
\indent Now, suppose that $G$ is infinite and $\Pi(G)$ contains a cycle $C$. It is easy to see that $C$ should contain a path of the form $M$ -- $H$ -- $N$, where $H$, $M$ and $N$ are subgroups of $G$ and furthermore $M$ and $N$ are maximal subgroups of $H$. If both $M$ and $N$ are normal subgroups of $H$, then $[M : M\cap N]=[MN : N]=[H : N]$ and similarly $[N : M\cap N]=[H : M]$. Thus $H$ -- $M$ -- $M\cap N$ -- $N$ -- $H$ is a 4-cycle in $\Pi(G)$. Now, assume that $M$ is not a normal subgroup of $H$.
Then $M$ -- $H$ -- $xMx^{-1}$ is a path in $\Pi(G)$, for some $x\in G$. Therefore, $M/\mathrm{Core}_H(M)$ -- $H/\mathrm{Core}_H(M)$ -- $xMx^{-1}/\mathrm{Core}_H(M)$ is a path in $\Pi(H/\mathrm{Core}_H(M))$. Clearly, $H/\mathrm{Core}_H(M)$ is a finite group which is not a cyclic $p$-group. So by the previous paragraph, $gr(\Pi(H/\mathrm{Core}_H(M)))=4$ and hence $gr(\Pi(G))=4$.
}
\end{proof}

By the proof of the previous theorem, we have the following corollary.

\begin{cor}
If $G$ is a finite group or an infinite abelian group, then $\Pi(G)$ is a forest if and only if $G$ is isomorphic to either $\mathbb{Z}_{p^n}$ or $\mathbb{Z}{(p^\infty)}$, where $p$ is a prime and $n$ is a positive integer.
\end{cor}

\begin{proof}
{Suppose that $\Pi(G)$ is a forest. If $G$ is finite, then by the proof of Theorem \ref{girth}, $G\cong \mathbb{Z}_{p^n}$, for some prime number $p$ and positive integer $n$.
If $G$ is an infinite abelian group, then $G$ is a torsion $p$-group. (Note that $gr(\Pi(\mathbb{Z}))=4$ and if $G$ has two elements of orders $p$ and $q$, then $\mathbb{Z}_{pq}$ is a subgroup of $G$, where $p,q$ are distinct primes.) Also by the proof of Theorem \ref{girth}, every finite subgroup of $G$ is cyclic. Thus $G$ has no non-trivial direct summand. Now, by \cite[p.110]{rob}, $G\cong \mathbb{Z}{(p^\infty)}$, for some prime $p$. Clearly, $\Pi(\mathbb{Z}{(p^\infty)})$ is a disjoint union of an isolated vertex and an infinite path. The proof is complete.
}
\end{proof}

In the following theorem, we consider the prime index graph of cyclic groups.

\begin{thm}
\label{cyclic}
Let $n=p^{n_1}_1\cdots p^{n_s}_s$, where $p_1,\ldots,p_s$ are distinct primes and $n_1,\ldots,n_s$ are positive integers. Then $\Pi(\mathbb{Z}_n)\cong P_{n_1+1}\square \cdots \square P_{n_s+1}$.
\end{thm}

\begin{proof}
{We know that $\mathbb{Z}_n\cong \mathbb{Z}_{p^{n_1}_1}\times \cdots \times \mathbb{Z}_{p^{n_s}_s}$. If $H$ and $K$ are two distinct subgroups of $\mathbb{Z}_n$, then $H\cong \mathbb{Z}_{p^{\alpha_1}_1}\times \cdots \times \mathbb{Z}_{p^{\alpha_s}_s}$ and $K\cong \mathbb{Z}_{p^{\beta_1}_1}\times \cdots \times \mathbb{Z}_{p^{\beta_s}_s}$, where $0\leq \alpha_i, \beta_i\leq n_i$ for $i=1,\ldots,s$. So $H$ and $K$ are adjacent if and only if there exists an integer $j$, $1\leq j\leq s$, such that $\alpha_i=\beta_i$ for $i\neq j$ and $\alpha_j=\beta_j\pm 1$. Thus $\Pi(\mathbb{Z}_n)\cong \Pi(\mathbb{Z}_{p^{n_1}_1})\square \cdots \square \Pi(\mathbb{Z}_{p^{n_s}_s})$ and $\Pi(\mathbb{Z}_n)\cong P_{n_1+1}\square \cdots \square P_{n_s+1}$.
}
\end{proof}

\begin{thm}
\label{regular}
Let $G$ be a finite abelian group. If $\Pi(G)$ is regular, then $G\cong \mathbb{Z}_{p_1\cdots p_s}$ and $\Pi(G)\cong Q_s$, where $p_1,\ldots,p_s$ are distinct prime numbers.
\end{thm}

\begin{proof}
{Let $|G|=p^{\alpha_1}_1\cdots p^{\alpha_s}_s$, where $p_1,\ldots,p_s$ are distinct primes and $\alpha_1,\ldots,\alpha_s$ are positive integers. Assume that $G\cong \mathbb{Z}_{p^{\alpha_{11}}_1}\times \cdots \times \mathbb{Z}_{p^{\alpha_{1k_1}}_1}\times \cdots \times \mathbb{Z}_{p^{\alpha_{s1}}_s}\times \cdots \times \mathbb{Z}_{p^{\alpha_{sk_s}}_s}$, where $k_i$ is a positive integer and $\alpha_{i1}+\cdots+\alpha_{ik_i}=\alpha_i$, for $i=1,\ldots,s$. We claim that $k_i=1$ for each $i$, $1\leq i\leq s$. By contradiction assume that $k_i\neq 1$, for some $i$, $1\leq i\leq s$. Let $n(k_i,p_i)$ be the number of subgroups of order $p_i$ in $\mathbb{Z}_{p^{\alpha_{i1}}_i}\times \cdots \times \mathbb{Z}_{p^{\alpha_{ik_i}}_i}$. Obviously, the number of subgroups of order $p_i$ in two groups $\mathbb{Z}_{p^{\alpha_{i1}}_i}\times \cdots \times \mathbb{Z}_{p^{\alpha_{ik_i}}_i}$ and $(\mathbb{Z}_{p_i})^{k_i}$ are the same. Hence by \cite[p.59]{burn}, we have $n(k_i,p_i)=(p_i^{k_i}-1)/(p_i-1)$. Clearly, $d(\mathbb{Z}_{p^{\alpha_{i1}}_i})=1+n(k_i-1,p_i)+\sum_{j\neq i}n(k_j,p_j)$ and $d(\{e\})=\sum_{j=1}^{s}n(k_j,p_j)$. Since $\Pi(G)$ is a regular graph, so $n(k_i,p_i)=1+n(k_i-1,p_i)$. This implies that $p_i^{k_i-1}=1$ and hence $k_i=1$, a contradiction. The claim is proved. Thus $G$ is a cyclic group of order $p_1^{\alpha_1}\cdots p_s^{\alpha_s}$ and by Theorem \ref{cyclic}, $\Pi(G)\cong P_{\alpha_1+1}\square \cdots \square P_{\alpha_s+1}$. Now, since $\Pi(G)$ is a regular graph, $\alpha_1=\cdots=\alpha_s=1$ and $\Pi(G)\cong Q_s$.
}
\end{proof}

\begin{thm}
Let $G$ be a finite group. If $\Pi(G)$ is a $2$-regular graph, then $\Pi(G)\cong C_4$ and $G\cong \mathbb{Z}_{pq}$, where $p$ and $q$ are distinct primes.
\end{thm}

\begin{proof}
{Since $d(\{e\})=2$, the order of $G$ has at most two distinct prime divisors. Clearly, a $p$-group cannot have exactly two subgroups of order $p$. So assume that $p$ and $q$ are two distinct prime divisors of $|G|$. Suppose that $H$ and $K$ are subgroups of $G$ such that $|H|=p$ and $|K|=q$. Since $d(\{e\})=2$, $H$ and $K$ are normal subgroups of $G$. Hence $HK$ is a subgroup of $G$ and $\{e\}$ -- $H$ -- $HK$ -- $K$ -- $\{e\}$ is a cycle in $\Pi(G)$. Now since $\Pi(G)$ is a $2$-regular graph, $H$ is a Sylow $p$-subgroup and $K$ is a Sylow $q$-subgroup of $G$. Thus $G=HK\cong \mathbb{Z}_{pq}$ and $\Pi(G)\cong C_4$.
}
\end{proof}

\section{Connectivity}

In this section, we study those groups whose prime index graphs are connected. First we have the following lemma.

\begin{lem}
Let $G$ be an infinite group. Then $\Pi(G)$ is not connected. Moreover, if $G$ is a simple group, then $G$ is an isolated vertex in $\Pi(G)$.
\end{lem}

\begin{proof}
{It is clear that if $G$ is an infinite group, then there is no path between $\{e\}$ and $G$ in $\Pi(G)$. So $\Pi(G)$ is not connected. If $G$ is an infinite simple group, by \cite[Corollary 4.15]{rose}, $G$ cannot have a proper subgroup of finite index. Hence $G$ is an isolated vertex of $\Pi(G)$.
}
\end{proof}

By \cite[p.292]{rose}, a finite group $G$ is supersolvable if and only if each subgroup of $G$ satisfies the converse of Lagrange's Theorem. So for finite supersolvable groups $G$ such as finite abelian groups and finite $p$-groups, $\Pi(G)$ is connected. (Note that every subgroup of $G$ is connected to $\{e\}$.)

\begin{thm}
\label{abeli}
Let $G$ be a finite group and $N$ be a normal subgroup of $G$. If $\Pi(N)$ is a connected graph and  also for every subgroup $H/N$ of $G/N$, $\Pi(H/N)$ is a connected graph, then $\Pi(G)$ is connected.
\end{thm}

\begin{proof}
{Assume that $H$ is a subgroup of $G$. Hence $\Pi(HN/N)$ is a connected graph. Since $HN/N\cong H/(H\cap N)$, so the graph $\Pi(H/H\cap N)$ is connected. This implies that there is a path between $H$ and $H\cap N$ in $\Pi(G)$. Now, since $\Pi(N)$ is connected, there is a path between $H\cap N$ and $\{e\}$ in $\Pi(N)$. Thus every subgroup of $G$ is connected to $\{e\}$. Therefore $\Pi(G)$ is connected.
}
\end{proof}

Now, we prove that the prime index graph of every finite solvable group is connected.

\begin{thm}
\label{solvable}
Let $G$ be a finite solvable group. Then $\Pi(G)$ is connected.
\end{thm}

\begin{proof}
{Since $G$ is a solvable group, $G^{(n)}=\{e\}$, for some positive integer $n$. We prove the theorem by applying the induction on $n$. If $G'=\{e\}$, then $G$ is an abelian group and so $\Pi(G)$ is a connected graph. Assume that $n>1$ and $G^{(n)}=\{e\}$. By the induction hypothesis, $\Pi(G')$ is connected. Now, by Theorem \ref{abeli}, $\Pi(G)$ is a connected graph.
}
\end{proof}

If $G$ is a group of odd order, then $G$ is solvable (Feit-Thompson Theorem \cite{feit}) and by Theorem \ref{solvable}, $\Pi(G)$ is connected. Moreover, suppose that $|G|=2^nm$, where $m$ and $n$ are positive integers with $m$ odd. If $G$ has a cyclic Sylow $2$-subgroup, then by \cite[p.148]{Dummit}, $G$ has a normal subgroup of order $m$ and hence $G$ is a solvable group. Thus $\Pi(G)$ is a connected graph.
Since every subgroup of a solvable group is solvable, by Theorems \ref{abeli} and \ref{solvable}, we have the next result.

\begin{cor}
\label{sol}
Let $G$ be a finite group and $N$ be a normal subgroup of $G$. If $\Pi(N)$ is a connected graph and $G/N$ is a solvable group, then $\Pi(G)$ is connected. \end{cor}

\begin{thm}
\label{normal}
Let $G$ be a group and $N$ be a normal subgroup of $G$. If $\Pi(G)$ is a connected graph, then $\Pi(N)$ and $\Pi(G/N)$ are connected graphs.
\end{thm}

\begin{proof}
{First we prove that $\Pi(N)$ is a connected graph. Let $H$ and $K$ be two distinct subgroups of $N$. Since $\Pi(G)$ is a connected graph, so there is a path $H$ -- $L_1$ -- $\cdots$ -- $L_t$ -- $K$ from $H$ to $K$ in $\Pi(G)$. We claim that by removing the same consecutive vertices in $H$ -- $L_1\cap N$ -- $\cdots$ -- $L_t\cap N$ -- $K$ and keeping one of them we obtain a walk from $H$ to $K$ in $\Pi(N)$. With no loss of generality, assume that $L_i\subseteq L_{i+1}$ and $[L_{i+1} : L_i]=p$, for some prime number $p$. Thus $L_i\cap N\subseteq L_{i+1}\cap N$ and we have $$[L_{i+1}\cap N : L_i\cap N]=\frac{|L_{i+1}\cap N|}{|L_i\cap N|}=\frac{|L_iN|}{|L_{i+1}N|}\frac{|L_{i+1}|}{|L_i|}.$$ Hence $[L_{i+1}N : L_iN][L_{i+1}\cap N : L_i\cap N]=p$. Therefore $L_{i+1}\cap N=L_i\cap N$ or $[L_{i+1}\cap N : L_i\cap N]=p$. So the claim is proved. Hence there is a path from $H$ to $K$ in $\Pi(N)$ which implies that $\Pi(N)$ is connected. \\
\indent Next, assume that $H$ and $K$ are two distinct subgroups of $G$ containing $N$. Suppose that $H$ -- $L_1$ -- $\cdots$ -- $L_t$ -- $K$ is a path from $H$ to $K$ in $\Pi(G)$. Similar to the previous case, one can prove that $H/N$ -- $L_1N/N$ -- $\cdots$ -- $L_tN/N$ -- $K/N$ is a walk from $H/N$ to $K/N$ in $\Pi(G/N)$. Thus $\Pi(G/N)$ is also a connected graph.
}
\end{proof}

Now, we propose the following problem.

\noindent {\bf Problem.}
Let $G$ be a group and $N$ be a normal subgroup of $G$. If $\Pi(N)$ and $\Pi(G/N)$ are both connected, then is it true that $\Pi(G)$ is connected?

By Theorem \ref{normal}, we have the next corollary.

\begin{cor}
Let $G\cong H\times K$, for some groups $H$ and $K$. If $\Pi(G)$ is connected, then both $\Pi(H)$ and $\Pi(K)$ are connected.
\end{cor}

We close this article by the study of the connectivity of $\Pi(A_n)$ and $\Pi(S_n)$. Moreover, we show that the prime index graph of all groups up to 500 elements is connected except for $A_6$.

\noindent {\bf Remark.}
Let $n$ be a positive integer. Then $\Pi(A_n)$ is connected if and only if $n\leq 5$. Also, $\Pi(S_n)$ is a connected graph if and only if $n\leq 5$.
To prove the remark first assume that $n\leq 4$. Hence $A_n$ is a solvable group and by Theorem \ref{solvable}, $\Pi(A_n)$ is a connected graph. If $n=5$, we know that every proper subgroup of $A_5$ is solvable and $A_5$ contains a maximal subgroup of prime index, then $\Pi(A_5)$ is connected. Also if $n\leq 5$, since $A_n$ is a normal subgroup of $S_n$ and $\Pi(A_n)$ is connected, by Corollary \ref{sol}, $\Pi(S_n)$ is connected. Now, assume that $n>5$. If $n$ is not a prime number, then by \cite[p.305]{An}, $A_n$ has no subgroup of prime index and hence $A_n$ is an isolated vertex of $\Pi(A_n)$. Otherwise, if $H$ is a maximal subgroup of $A_n$ of prime index, then $H\cong A_{n-1}$ (see \cite[p.305]{An}). Since $n-1$ is not a prime number, so $\Pi(A_n)$ is not connected.
Thus by Theorem \ref{normal}, $\Pi(S_n)$ is not connected.

\begin{thm}
Let $G$ be a group and $|G|\leq 500$. If $\Pi(G)$ is not connected, then $G\cong A_6$.
\end{thm}

\begin{proof}
{Suppose that $G$ is the smallest group such that $\Pi(G)$ is not a connected graph. By Theorem \ref{abeli}, one can see that $G$ is a simple group. Note that by the remark, $\Pi(A_6)$ is not a connected graph. On the other hand by \cite[p.295]{rotman}, if $G$ is a non-abelian simple group of order at most 500, then $G$ is isomorphic to one of the groups $A_5$, PSL$(2,7)$, or $A_6$. By remark, $\Pi(A_5)$ is connected. Also by \cite[Theorem 6.26]{suzuki}, PSL$(2,7)$ contains a maximal subgroup of index 7 and by \cite[p.191]{miller}, all subgroups of PSL$(2,7)$ are solvable. Hence $\Pi(\rm{PSL}(2,7))$ is connected. Thus $G\cong A_6$. Finally, for every non-abelian group $G$ with $360<|G|\leq 500$, since $G$ is not a simple group, so $G$ has a non-trivial proper normal subgroup $N$. Clearly, $|N|$ and $|G/N|$ are both less than 360. Thus by Theorem \ref{abeli}, $\Pi(G)$ is a connected graph.
}
\end{proof}

\end{document}